\documentclass[12pt,reqno]{amsart}

\usepackage{amssymb}

\newtheorem{theorem}{Theorem}[section]
\newtheorem{proposition}[theorem]{Proposition}
\newtheorem{lemma}[theorem]{Lemma}

\theoremstyle{definition}

\numberwithin{equation}{section}

\begin{document}
\baselineskip=15.5pt

\title[Deformation quantization, translation structure and Higgs bundles]{Deformation quantization of moduli spaces 
of Higgs bundles on a Riemann surface with translation structure}

\author[I. Biswas]{Indranil Biswas}

\address{School of Mathematics, Tata Institute of Fundamental
Research, Homi Bhabha Road, Mumbai 400005, India}

\email{indranil@math.tifr.res.in}



\date{}

\begin{abstract}
Let $X$ be a compact connected Riemann surface of genus $g\, \geq\, 1$ equipped with a nonzero
holomorphic $1$-form. Let ${\mathcal M}_X(r)$ denote the moduli space of semistable 
Higgs bundles on $X$ of rank $r$ and degree $r(g-1)+1$; it is a complex symplectic manifold. Using
the translation structure on the open subset of $X$ where the $1$-form does not vanish, we construct a natural
deformation quantization of a certain nonempty Zariski open subset of ${\mathcal M}_X(r)$.
\end{abstract}

\maketitle

\section{Introduction}

Let $X$ be a compact connected Riemann surface of genus $g$, with $g\, \geq\, 1$. Take a holomorphic $1$-form
$$
\beta\,\,\in\,\, H^0(X,\, K_X)\setminus\{0\}\,.
$$
So $\beta$ defines a translation structure on the Zariski open subset
$$
X_0\, :=\, \left\{x\, \in\, X\ \big\vert\ \beta(x)\, \not=\, 0\right\}\, \subset\, X\, .
$$
This means that $X_0$ is covered by a distinguished class of holomorphic coordinate functions such that all the 
transition functions are translations of $\mathbb C$. Consider the total space of the holomorphic cotangent bundle 
$K_{X_0}$ of $X_0$ equipped with the Liouville symplectic form. Using the translation structure on $X_0$, we construct a 
deformation quantization of this symplectic manifold (see Proposition \ref{prop1}); the definition of a deformation 
quantization is recalled in Section \ref{sec2.1} (more details on deformation quantization can be found in 
\cite{BFFLS}, \cite{DWL} and references therein).

It may be mentioned that if $X$ is equipped with a projective structure, then the complement $K_X\setminus \{0_X\}$ 
of the zero section of $K_X$, equipped with the Liouville symplectic form, has a natural deformation quantization 
\cite{BB}.

For any integer $n\, \geq\, 1$, consider
$$
K^n_0\, :=\, \left\{(x_1,\, \cdots,\, x_n)\, \in\, (K_X)^n\,\, \middle|\,\, x_i\, \not=\, x_j\ \ 
\forall\,\ i\, \not=\, j\right\}\, ;
$$
the Liouville symplectic form on $K_X$ produces a holomorphic symplectic form on $K^n_0$,
which is denoted by ${\Omega}'_K$. The above mentioned
deformation quantization of $K_{X_0}$ produces a deformation quantization of the Zariski open subset
$K^n_0\bigcap (K_{X_0})^n\, \subset\, (K_X)^n$ equipped with the symplectic form
${\Omega}'_K\big\vert_{K^n_0\cap (K_{X_0})^n}$.

The symplectic form ${\Omega}'_K$ on $K^n_0$ is preserved under the natural action of the symmetric group 
${\mathbb S}_n$ of permutations of $\{1,\, \cdots,\, n\}$. Since the action of ${\mathbb S}_n$ on $K^n_0$ is free, we 
get a holomorphic symplectic form on the manifold $K^n_0/{\mathbb S}_n$ given by ${\Omega}'_K$,
which is denoted by ${\Omega}''_K$. The 
deformation quantization of $K^n_0\bigcap (K_{X_0})^n$ is also invariant under the natural action of ${\mathbb S}_n$ 
on $K^n_0\bigcap (K_{X_0})^n$. Consequently, it produces a deformation quantization of the symplectic structure
${\Omega}''_K\big\vert_{(K^n_0\cap (K_{X_0})^n)/{\mathbb S}_n}$ on $(K^n_0\bigcap (K_{X_0})^n)/{\mathbb S}_n$.

For any integer $r\, \geq\, 1$, let ${\mathcal M}_X(r)$ denote the moduli space of semistable
Higgs bundles on $X$ of rank $r$ and degree $r(g-1)+1$; it is a smooth complex quasiprojective variety of
dimension $2(r^2(g-1)+1)$ equipped with an algebraic symplectic form. We will denote the symplectic
form on ${\mathcal M}_X(r)$ by $\Omega_M$.

Set $\delta\,=\, r^2(g-1)+1$. The moduli space ${\mathcal M}_X(r)$ is birational to $K^\delta_0/{\mathbb S}_\delta$
\cite{GNR}, \cite{ER}. More precisely, there is a nonempty Zariski open subset
$$\widehat{\mathcal U}_M\, \subset\, {\mathcal M}_X(r)$$ and a nonempty Zariski open subset
$$\widehat{\mathcal U}_S\, \subset\, (K^n_0\cap (K_{X_0})^n)/{\mathbb S}_n$$
together with a natural algebraic isomorphism
$$
\widehat{\Phi}\, \, :\,\,
\widehat{\mathcal U}_M \, \stackrel{\sim}{\longrightarrow}\, \widehat{\mathcal U}_S
$$
such that
$$
\widehat{\Omega}_M\, :=\, \Omega_M\big\vert_{\widehat{\mathcal U}_M}\,=\,
\widehat{\Phi}^*\widehat{\Omega}_S\, ,
$$
where $\widehat{\Omega}_S\, :=\, {\Omega}''_K\big\vert_{\widehat{\mathcal U}_S}$.

The above mentioned deformation quantization of the symplectic
manifold 
$$
\left((K^\delta_0\cap (K_{X_0})^\delta)/{\mathbb S}_\delta,\, {\Omega}''_K\big\vert_{(K^\delta_0\cap 
(K_{X_0})^\delta)/{\mathbb S}_\delta}\right)
$$
produces a deformation quantization of the symplectic manifold $(\widehat{\mathcal U}_S,\, 
\widehat{\Omega}_S)$. Using the above isomorphism $\widehat{\Phi}$, this produces a deformation quantization of the 
symplectic manifold $(\widehat{\mathcal U}_M,\, \widehat{\Omega}_M)$; see Theorem \ref{thm1}.

\section{Constant symplectic form and its quantization}\label{se2}

\subsection{Deformation quantization}\label{sec2.1}

Let $Y$ be a connected complex manifold. Its holomorphic tangent and cotangent bundles will be denoted
by $TY$ and $T^*Y$ respectively. Assume that $Y$ is equipped with a holomorphic symplectic form $\Theta$.
The closed $2$--form $\Theta$ defines a holomorphic homomorphism
$$\eta\, :\, TY\, \longrightarrow\, T^*Y$$ that sends any $v\, \in\, T_xY$ to $-i_v(\Theta(x))\, \in\, T^*_xY$,
so $\eta(x)(w)(v)\,=\, \Theta(v,\, w)$ for all $v,\, w\, \in\, T_xY$ and $x\, \in\, Y$. This homomorphism
$\eta$ is an isomorphism, because $\Theta$ is nondegenerate. Define
$$
\tau\, :=\, \eta^{-1}\, :\, T^*Y\, \longrightarrow\, TY\, .
$$
For any two holomorphic functions $f_1,\, f_2$ defined on some open subset $U\, \subset\, Y$, define the
holomorphic function on $U$
\begin{equation}\label{s2}
\{f_1,\, f_2\}\, :=\, \Theta(\tau(df_1),\, \tau(df_2))\, .
\end{equation}
We have
\begin{itemize}
\item $\{f_1,\, f_2\}\, =\, -\{f_2,\, f_1\}$,

\item $\{f_1,\, f_2f_3\}\, =\, \{f_1,\, f_2\}\cdot f_3 + \{f_1,\, f_3\}\cdot f_2$, and

\item $\{f_1,\, \{f_2,\, f_3\}\}+ \{f_2,\, \{f_3,\, f_1\}\}+ \{f_3,\, \{f_1,\, f_2\}\}\,=\, 0$.
\end{itemize}
So the operation $\{-,\, -\}$ defines a Poisson structure on $Y$.

The algebra of locally defined holomorphic functions on $Y$ will be denoted by ${\mathcal H}(Y)$. Let
$$
{\mathcal A}(Y)\,\, :=\,\, {\mathcal H}(Y)[[h]]\,\, :=\,\, \left\{\sum_{i=0}^\infty h^if_i\,\,\middle|\,\,
f_i\, \in\, {\mathcal H}(Y)\right\}
$$
be the space of all formal power series.

Consider an associative algebra operation
$$
{\mathcal A}(Y)\times {\mathcal A}(Y) \, \longrightarrow\,{\mathcal A}(Y)\, . 
$$
The image of any pair $(\textbf{f},\, \textbf{g})$, where $\textbf{f}\,=\, \sum_{i=0}^\infty h^if_i$ and
$\textbf{g}\,=\, \sum_{i=0}^\infty h^ig_i$, is an element
\begin{equation}\label{s1}
\textbf{f}\star\textbf{g}\,=\, \sum_{i=0}^\infty h^i \alpha_i\, \in\, {\mathcal A}(Y)\, .
\end{equation}
A \textit{deformation quantization} of the Poisson manifold $(Y,\, \{-,\, -\})$ is an associative algebra operation 
``$\star$'' as in \eqref{s1} satisfying the following conditions:
\begin{enumerate}
\item each $\alpha_i$ in \eqref{s1} is some polynomial in the derivatives, of arbitrary order, of
$\{f_i\}_{i\geq 0}$ and $\{g_i\}_{i\geq 0}$ (it should be emphasized that the polynomial $\alpha_i$
itself is independent of $\textbf{f}$ and $\textbf{g}$),

\item $\alpha_0\,=\, f_0g_0$,

\item $1\star t \,=\, t\,=\, t\star 1$ for all $t\, \in\, {\mathcal H}(Y)$, and

\item $\textbf{f}\star\textbf{g}- \textbf{g}\star\textbf{f}\,=\, \sqrt{-1}h\{f_0,\, g_0\}+h^2\gamma$, where
$\gamma\, \in\, {\mathcal A}(Y)$ (it depends on $\textbf{f},\, \textbf{g}$).
\end{enumerate}
(See \cite{BFFLS}, \cite{DWL}, \cite{Fe}, \cite{We} for more details.)

\subsection{Moyal--Weyl deformation quantization}\label{sec2.2}

We will now recall an explicit deformation quantization of a constant symplectic form on a vector space.

Let $V$ be a complex vector space of even dimension, say $2n$. Fix a constant symplectic
form
$$
\Theta_0\, \in\, \bigwedge\nolimits^2 V^*
$$
on $V$; in other words, $\Theta_0$ is nondegenerate. Following the notation of Section \ref{sec2.1},
the space of all locally defined holomorphic functions on $V$ (respectively, $V\times V$) is denoted by
${\mathcal H}(V)$ (respectively, ${\mathcal H}(V\times V)$). The form
$\Theta_0$ defines a Poisson structure on ${\mathcal H}(V)$; see \eqref{s2}. As before, the Poisson structure
will be denoted by $\{-,\, -\}$. Define as before
$$
{\mathcal A}(V)\,:=\, {\mathcal H}(V)[[h]]\,\, :=\,\, \left\{\sum_{i=0}^\infty h^if_i\,\, \middle|\,\,
f_i\, \in\, {\mathcal H}(V)\right\}.
$$

For any $f_1,\, f_2\, \in\, {\mathcal H}(V)$, the element of
${\mathcal H}(V\times V)$ defined by $(v,\, w)\, \longmapsto\, f_1(v)\cdot f_2(w)$, where
$v,\, w\, \in\, V$, will be denoted by $f_1\otimes f_2$. Let
$$
\Delta\,:\, V\, \longrightarrow\, V\times V\, , \ \ v\, \longmapsto\, (v,\, v)
$$
be the diagonal embedding. It defines a homomorphism
$$
\Delta^*\,:\, {\mathcal H}(V\times V)\, \longrightarrow\,{\mathcal H}(V)\, ,\ \ \Delta^* (f)(v)\, =\,
f(v,\, v)\, .
$$ 

There is a unique differential operator with constant coefficients
\begin{equation}\label{d}
D\ :\, {\mathcal H}(V\times V)\, \longrightarrow\,{\mathcal H}(V\times V)
\end{equation}
that satisfies the following condition: for any $f_1,\, f_2\, \in\, {\mathcal H}(V)$,
$$
\{f_1,\, f_2\}\,=\, \Delta^* D(f_1\otimes f_2)
$$
(see \cite{Fe}, \cite{We}).

The \textit{Moyal--Weyl deformation quantization} of the Poisson manifold $({\mathcal H}(V),\, \{-,\, -\})$ is defined
by the following conditions:
\begin{enumerate}
\item for all $f_1,\, f_2\, \in\, {\mathcal H}(V)$,
\begin{equation}\label{d2}
f_1\star f_2\,=\, \Delta^* \exp(\sqrt{-1}hD/2)(f_1\otimes f_2)\, \in\, {\mathcal A}(Y)\, ,
\end{equation}
where $D$ is defined in \eqref{d}, and

\item extend the above multiplication operation ``$\star$'' to ${\mathcal A}(V)$ using the bilinearity
condition with respect to the formal parameter $h$. In other words, for
$\textbf{f},\, \textbf{g}\, \in\, {\mathcal A}(V)$ with $\textbf{f}\,=\, \sum_{i=0}^\infty h^if_i$ and
$\textbf{g}\,=\, \sum_{i=0}^\infty h^ig_i$, define
$$
\textbf{f}\star\textbf{g}\,=\, \sum_{i=0}^\infty h^i \left(\sum_{j=0}^i f_j\star g_{i-j}\right)\, \in\, {\mathcal A}(V)\, .
$$
\end{enumerate}
This multiplication operation ``$\star$'' is a deformation quantization of the Poisson manifold
$({\mathcal H}(V),\, \{-,\, -\})$ (see \cite{We}, \cite{Fe} and references therein). It is known as
the \textit{Moyal--Weyl deformation quantization}.

Take any $v_0\, \in\, V$. Define the translation map
$$
{\mathcal T}_0\, :\, V\, \longrightarrow\, V\, , \ \ v\, \longmapsto\, v+v_0\, .
$$
Let
$$
{\mathcal T}\, :\, {\mathcal A}(V)\, \longrightarrow\, {\mathcal A}(V)
$$
be the automorphism that sends any $\textbf{f}\,=\, \sum_{i=0}^\infty h^if_i$ to
$\sum_{i=0}^\infty h^i(f_i\circ {\mathcal T}_0)\, \in\, {\mathcal A}(V)$.

\begin{lemma}\label{lem1}
For any $\textbf{f},\, \textbf{g}\, \in\, {\mathcal A}(V)$,
$$
{\mathcal T}(\textbf{f}\star\textbf{g})\,=\, {\mathcal T}(\textbf{f})\star {\mathcal T}(\textbf{g})\, .
$$
\end{lemma}

\begin{proof}
Recall that the differential operator $D$ in \eqref{d} has constant coefficients. From this
it follows immediately that
$$
(f_1\circ {\mathcal T}_0)\star (f_2\circ {\mathcal T}_0)\,=\, {\mathcal T}(f_1\star f_2)
$$
for all $f_1,\, f_2\, \in\, {\mathcal H}(V)$ (see \eqref{d2}). The lemma follows from this.
\end{proof}

It may be mentioned that the Moyal--Weyl deformation quantization has the important property that it is invariant
under the action of the symplectic group ${\rm Sp}(V,\,\Theta_0)\, \subset\, {\rm GL}(V)$ on $V$, however,
we won't need it here.

\section{One-form on Riemann surfaces}\label{se3}

Let $X$ be a compact connected Riemann surface of genus $g$, with $g\, \geq\,1$.
The holomorphic cotangent bundle of $X$ will be denoted by $K_X$.
Fix a nonzero holomorphic $1$-form
\begin{equation}\label{e1}
\beta\,\,\in\,\, H^0(X,\, K_X)\setminus\{0\}
\end{equation}
Consider the zero-set
$$
Z_\beta\, :=\, \left\{x\, \in\, X\,\,\middle|\,\, \beta(z)\,=\, 0\right\}\, \subset\, X\, .
$$
It's cardinality is at most $\text{degree}(K_X)\,=\, 2g-2$, and $Z_\beta\,=\, \emptyset$ if $g\,=\, 1$. Let
\begin{equation}\label{e2}
X_0\,\, :=\,\, X\setminus Z_\beta
\end{equation}
be the complement.

Take any simply connected open subset $U\, \subset\ X_0$. There is a holomorphic function
$f_U$ on $U$ such that $df_U\,=\, \beta\big\vert_U$. This condition uniquely determines $f_U$
up to an additive constant, meaning for any $\widetilde{f}_U$ with $d\widetilde{f}_U\,=\, \beta\big\vert_U$,
there is a constant $c\, \in\, \mathbb C$ such that
\begin{equation}\label{t1}
\widetilde{f}_U\,=\, f_U+c \,.
\end{equation}
Since $\beta$ does not vanish on $U$, the above function $f_U$ is an immersion. So for every $x\, \in\, U$, there is 
an open neighborhood $x\, \in\, U_x\, \subset\, U$ such that the restriction $\widetilde{f}_U\big\vert_{U_x}$ is an 
embedding. Hence $\widetilde{f}_U$ is a holomorphic coordinate function on $U_x$. As $x$ runs over points on $X_0$, 
we get a holomorphic coordinate atlas on $X_0$ satisfying the condition that all the transition functions are 
translations of $\mathbb C$. Such a holomorphic coordinate atlas defines a \textit{translation structure} on $X_0$. 
So $\beta$ defines a translation structure on $X_0$ (see \cite{EMZ} and references therein for translation 
structure). In fact, $\beta$ defines a branched translation structure on entire $X$ (see \cite{LF}, \cite{Ka}, 
\cite{BJJP}, \cite{De}); see \cite{BD} for general branched structures.

On ${\mathbb C}^2$ we have the standard constant symplectic form
\begin{equation}\label{e3}
\omega_0\, :=\, (dz_2)\wedge (dz_1)\, .
\end{equation}
Identify ${\mathbb C}^2$ with the total space of the holomorphic cotangent bundle $T^*{\mathbb C}$
by sending $u\cdot dz\, \in\, T^*_c{\mathbb C}$ to $(c,\, u)\,\in\, {\mathbb C}^2$. Under this
identification of ${\mathbb C}^2$ with $T^*{\mathbb C}$, the form $\omega_0$ in \eqref{e3} gets
identified with the Liouville symplectic form on $T^*{\mathbb C}$.

Take a pair $(U,\, f_U)$, where $U$ is a connected open subset of $X_0$ and
$f_U\, :\, U\, \longrightarrow\, {\mathbb C}$ is a holomorphic embedding with $df_U\,=\, \beta\big\vert_U$. Let
\begin{equation}\label{e4}
F\,:\, K_U\,=\, K_X\big\vert_U \,=\, T^*U\, \longrightarrow\, T^*{\mathbb C}\,=\, {\mathbb C}^2
\end{equation}
be the holomorphic embedding given by the differential of $f_U$, more precisely, $$F(u,\, \lambda \cdot \beta(u))
\,=\, (f_U(u),\, \lambda)$$ for all $u\, \in\, U$ and $\lambda\, \in\, \mathbb C$.

\begin{lemma}\label{lem2}
The pullback $F^*\omega_0$ of the form $\omega_0$ in \eqref{e3} coincides with the Liouville symplectic
form on $K_U\,=\, T^*U$.
\end{lemma}

\begin{proof}
This follows from the fact that the earlier mentioned identification
of ${\mathbb C}^2$ with $T^*{\mathbb C}$ takes the form $\omega_0$
to the Liouville symplectic form on $T^*{\mathbb C}$.
\end{proof}

Consider the Moyal--Weyl deformation quantization of the symplectic manifold $$(T^*{\mathbb C},\, \omega_0)\,=\,
({\mathbb C}^2,\, \omega_0)\, ,$$
where $\omega_0$ is the constant form in \eqref{e3}. In view
of Lemma \ref{lem2}, using the holomorphic embedding $F$ in \eqref{e4}, the
Moyal--Weyl deformation quantization of $({\mathbb C}^2,\, \omega_0)$
produces a deformation quantization of the symplectic manifold $T^*U\,=\, K_U$ equipped with the Liouville symplectic form.
To describe this deformation quantization explicitly, let
$$
F^{-1}\, :\, F(K_U)\, \longrightarrow\, K_U
$$
be the inverse of $F$ on the image of $F$. For
$\textbf{f},\, \textbf{g}\, \in\, {\mathcal A}(K_U)$ with $\textbf{f}\,=\, \sum_{i=0}^\infty h^if_i$ and
$\textbf{g}\,=\, \sum_{i=0}^\infty h^ig_i$, where $f_i,\, g_i\, \in\, {\mathcal H}(K_U)$, if
$$
\left(\sum_{i=0}^\infty h^i(f_i\circ F^{-1})\right)\star \left(\sum_{i=0}^\infty h^i(g_i\circ F^{-1})\right)\,=\,
\sum_{i=0}^\infty h^i b_i\, ,
$$
where $\star$ is the Moyal--Weyl deformation quantization and $b_i\, \in\, {\mathcal H}(F(K_U))$, then define
\begin{equation}\label{dq}
\textbf{f}\star\textbf{g}\,=\, \sum_{i=0}^\infty h^i (b_i\circ F)\,.
\end{equation}
Using Lemma \ref{lem2} we conclude that
this defines a deformation quantization of the symplectic manifold $T^*U$ equipped with the Liouville symplectic form.

The holomorphic coordinate function $f_U$ is not uniquely determined by $\beta$. However, any two choices of the 
holomorphic coordinate function are related by \eqref{t1}. Using Lemma \ref{lem1} and \eqref{t1} we conclude that 
the deformation quantization in \eqref{dq}, of the symplectic manifold $K_U$, equipped with the Liouville symplectic 
form, is actually independent of the choice the function $f_U$. Consequently, the locally defined deformation 
quantizations of the symplectic manifold $T^*X_0\,=\, K_{X_0}$ equipped with the Liouville symplectic form patch 
together compatibly to define a deformation quantization of the symplectic manifold $K_{X_0}$ equipped with the 
Liouville symplectic form.

we summarize the above construction in the following proposition:

\begin{proposition}\label{prop1}
Given a nonzero holomorphic $1$-form $\beta$ on $X$, the symplectic manifold $T^*X_0\,=\, K_{X_0}$ equipped with the 
Liouville symplectic form has a natural deformation quantization. It is locally given by the Moyal--Weyl deformation 
quantization of the symplectic manifold $(T^*{\mathbb C},\, \omega_0)$, where $\omega_0$ is the constant form in 
\eqref{e3}. These locally defined deformation quantizations patch together compatibly to define a global deformation 
quantization.
\end{proposition}

\section{Moduli space of Higgs bundles and deformation quantization}

\subsection{A description of moduli spaces of Higgs bundles}

A Higgs bundle on $X$ is a holomorphic vector bundle $E$ on $X$ together with a holomorphic section
$$\theta\, \in\, H^0(X,\, \text{End}(E)\otimes K_X)$$ \cite{Hi}, \cite{Si1}. A Higgs bundle $(E,\, \theta)$
is called \textit{semistable} if
$$
\frac{{\rm degree}(F)}{{\rm rank}(F)}\, \leq\, \frac{{\rm degree}(E)}{{\rm rank}(E)}
$$
for all holomorphic subbundles $0\, \not=\,F\, \subset\, E$ with $\theta(F)\, \subset\, F\otimes K_X$. If
$$
\frac{{\rm degree}(F)}{{\rm rank}(F)}\, <\, \frac{{\rm degree}(E)}{{\rm rank}(E)}
$$
for all holomorphic subbundles $0\, \not=\, F\, \subsetneq\, E$ with $\theta(F)\, \subset\, F\otimes K_X$,
then $(E,\, \theta)$ is called \textit{stable}.

Fix an integer $r\, \geq\, 1$. Let $\mathcal{M}_X(r)$ denote the moduli space of semistable Higgs bundles on
$X$ of rank $r$ and degree $d\, =\, r(g-1)+1$, where $g$ as before is
the genus of $X$ (see \cite{Hi}, \cite{Ni}, \cite{Si2} for the construction of
this moduli space). Since $r$ and $d$ are coprime, any $(E,\, \theta)\, \in\, \mathcal{M}_X(r)$ is in fact stable.
Consequently, $\mathcal{M}_X(r)$ is an irreducible smooth quasiprojective variety defined over $\mathbb C$.
Its dimension is
\begin{equation}\label{delta}
2\delta\,:=\, 2(r^2(g-1)+1)\, .
\end{equation}
Moreover, it has a natural algebraic symplectic structure \cite{Hi}.

Consider the Zariski open subset of the Cartesian product
\begin{equation}\label{kd}
K^{\delta}_0\, :=\, \left\{(x_1,\, \cdots,\, x_\delta)\, \in\, (K_X)^\delta\,\, \middle|\,\, x_i\, \not=\, x_j\ \ 
\forall\,\ i\, \not=\, j\right\}
\, \subset\, (K_X)^\delta
\end{equation}
given by the locus of all distinct $\delta$ points, where
$\delta$ is the integer in \eqref{delta}. The symmetric group ${\mathbb S}_\delta$ of permutations of 
$\{1,\, \cdots, \, \delta\}$ acts freely on $K^{\delta}_0$. So the quotient $K^{\delta}_0/{\mathbb S}_\delta$ is an 
irreducible smooth complex quasiprojective variety of dimension $2\delta$.

Let
\begin{equation}\label{e7}
p\, :\, K_X \, \longrightarrow\, X
\end{equation}
be the natural projection.

There is a natural nonempty Zariski open subset 
\begin{equation}\label{e5}
{\mathcal U}_M\, \subset\, \mathcal{M}_X(r)
\end{equation}
and a natural nonempty Zariski open subset
\begin{equation}\label{e6}
{\mathcal U}_S\, \subset\, K^{\delta}_0/{\mathbb S}_\delta
\end{equation}
such that there is a canonical algebraic isomorphism
\begin{equation}\label{Ph}
\Phi\, :\, {\mathcal U}_M \, \stackrel{\sim}{\longrightarrow}\, {\mathcal U}_S
\end{equation}
\cite{GNR}, \cite{ER}, \cite{BM}, \cite{Hu}. We will briefly recall it below.

Take any $(E,\, \theta)\, \in\, \mathcal{M}_X(r)$. Since $\text{degree}(E)\,=\, d\,=\, {\rm rank}(E)(g-1)+1$, from
Riemann--Roch theorem we know that $\dim H^0(X,\, E) - \dim H^1(X,\, E)\,=\, 1$. There is a nonempty Zariski
open subset
\begin{equation}\label{u0}
U_0\, \subset\, \mathcal{M}_X(r)
\end{equation}
such that $\dim H^0(X,\, F)\,=\, 1$ for all $(F,\, \theta)\, \in\, U_0$.
Assume that $(E,\, \theta)\, \in\, U_0$. For $(E,\, \theta)$, we have
\begin{itemize}
\item a $1$-dimensional closed subscheme $C\, \subset\, K_X$ such that the projection 
$$
\phi\, :=\, p\big\vert_{C} \, :\, C\, \longrightarrow\, X
$$
(see \eqref{e7} for the map $p$) is a finite map, and a 

\item a torsionfree coherent sheaf ${\mathbb L}\, \longrightarrow\, C$ of rank one,
\end{itemize}
satisfying the condition that $E\,=\, \phi_*{\mathbb L}$ \cite{Hi}. The Higgs field $\theta$ on $E$
is constructed from $(C,\, {\mathbb L})$ as follows. Let
$$
\sigma\, \in\, H^0(K_X,\, p^*K_X)
$$
be the tautological section, where $p$ is the projection in \eqref{e7}. Tensoring with it produces a homomorphism
$$
{\mathbb L} \, \xrightarrow{\, \ -\otimes \sigma\ \,}\, {\mathbb L}\otimes p^*K_X
$$
Taking the direct image and invoking the projection formula, we get a homomorphism
$$
E\,=\, \phi_*{\mathbb L} \, \xrightarrow{\ \phi_*(-\otimes \sigma)\ }\, \phi_*({\mathbb L}\otimes p^*K_X)
\,=\, (\phi_*{\mathbb L})\otimes K_X\,=\, E\otimes K_X\, .
$$
This homomorphism coincides with $\theta$ \cite{Hi}, \cite{Si2}.

Since $E\,=\, \phi_*{\mathbb L}$, and $\phi$ is a finite map, we have
$$
H^i(X,\, E)\,=\, H^i(C,\, {\mathbb L})
$$
for all $i$. In particular, we have $H^0(X,\, E)\,=\, H^0(C,\, {\mathbb L})$. So
$$
\dim H^0(C,\, {\mathbb L})\,=\, \dim H^0(X,\, E)\,=\, 1\, .
$$
Take a nonzero section $s\, \in\, H^0(C,\, {\mathbb L})\setminus\{0\}$. Let $D_s\,=\, \text{div}(s)$ denote
its divisor; since $\dim H^0(C,\, {\mathbb L})\,=\,1$, we know that $D_s$ is
actually independent of the choice of $s$.

There is a nonempty Zariski open subset
$$
{\mathcal U}_M\, \subset\, U_0
$$
(see \eqref{u0} for $U_0$) such that for all $(E,\, \theta)\, \in\, {\mathcal U}_M$, we have
$$
D_s\, \in\, K^{d}_0/{\mathbb S}_\delta
$$
(see \eqref{kd}). The isomorphism $\Phi$ in \eqref{Ph} sends any
$(E,\, \theta)\, \in\, {\mathcal U}_M$ to $D_s\, \in\, K^{d}_0/{\mathbb S}_\delta$. See
\cite{GNR}, \cite{ER} for more details.

The Liouville symplectic form on $K_X$ produces an algebraic symplectic form on the
Cartesian product $(K_X)^\delta$; let
\begin{equation}\label{sr}
\widetilde{\Omega}'_K\, \in\, H^0\left((K_X)^\delta,\, \Omega^2_{(K_X)^\delta}\right)
\end{equation}
be this symplectic form on $(K_X)^\delta$. Let
\begin{equation}\label{srn}
\Omega'_K\, \in\, H^0\left(K^{\delta}_0,\, \Omega^2_{K^{\delta}_0}\right)
\end{equation}
be the restriction of $\widetilde{\Omega}'_K$ to the
Zariski open subset $K^{\delta}_0$ in \eqref{kd}.
This $2$-form $\Omega'_K$ is evidently preserved by the action of the symmetric
group ${\mathbb S}_\delta$ on $K^{\delta}_0$. So $\Omega'_K$ produces an algebraic
symplectic form $\Omega''_K$ on the quotient space $K^{\delta}_0/{\mathbb S}_\delta$. Let
\begin{equation}\label{fo1}
\Omega_S\, \,\in\, \, H^0\left({\mathcal U}_S,\, \Omega^2_{{\mathcal U}_S}\right)
\end{equation}
be the restriction of $\Omega''_K$ to the open subset ${\mathcal U}_S$ in \eqref{e6}.

Let $\Omega'_M$ denote the algebraic symplectic form on the moduli space
$\mathcal{M}_X(r)$. Let
\begin{equation}\label{fo2}
\Omega_M\, \,\in\, \, H^0\left({\mathcal U}_M,\, \Omega^2_{{\mathcal U}_M}\right)
\end{equation}
be its restriction to the open subset ${\mathcal U}_M$ in
\eqref{e5}. For the isomorphism $\Phi$ in \eqref{Ph}, we have
\begin{equation}\label{e8}
\Phi^* \Omega_S\,\,=\,\,\Omega_M\, ,
\end{equation}
where $\Omega_S$ and $\Omega_M$ are defined in \eqref{fo1} and \eqref{fo2} respectively
(see \cite{GNR}, \cite{ER}, \cite{BM}, \cite{Hu}).

\subsection{Deformation quantization of the moduli space}

As in \eqref{e1}, fix a nonzero $1$-form
$$
\beta\,\,\in\,\, H^0(X,\, K_X)\setminus\{0\}\, .
$$

In Proposition \ref{prop1} we constructed a deformation quantization of
the symplectic manifold $T^*X_0\,=\, K_{X_0}$ equipped with the Liouville symplectic form.
Restrict the symplectic form $\widetilde{\Omega}'_K$ in \eqref{sr} to the open subset
$$
(K_{X_0})^\delta \, \subset\, (K_X)^\delta\, .
$$
The deformation quantization of $K_{X_0}$ produces a deformation quantization of this
symplectic manifold $\left((K_{X_0})^\delta,\, \widetilde{\Omega}'_K\big\vert_{(K_{X_0})^\delta}\right)$. From the
construction of this deformation quantization
of the symplectic manifold $\left((K_{X_0})^\delta,\, \widetilde{\Omega}'_K\big\vert_{(K_{X_0})^\delta}\right)$
it is evident that the permutation action of the symmetric
group ${\mathbb S}_\delta$ on $(K_{X_0})^\delta$ preserves this deformation quantization.

The deformation quantization of the symplectic manifold $\left((K_{X_0})^\delta,\, 
\widetilde{\Omega}'_K\big\vert_{(K_{X_0})^\delta}\right)$ in turn restricts to a deformation quantization of the
Zariski open subset $(K_{X_0})^\delta\bigcap K^{\delta}_0\, \subset\, K^{\delta}_0$ equipped with the symplectic
form ${\Omega}'_K\big\vert_{(K_{X_0})^\delta\cap K^{\delta}_0}$, where $K^{\delta}_0$ is constructed in \eqref{kd}
and ${\Omega}'_K$ is the symplectic form in \eqref{srn}. Let
$$
\widehat{K^{\delta}_0/{\mathbb S}_\delta}\, \subset\, K^{\delta}_0/{\mathbb S}_\delta
$$
be the Zariski open subset given by the image of $(K_{X_0})^\delta\bigcap K^{\delta}_0$. From the above observation, 
that the permutation action of the symmetric group ${\mathbb S}_\delta$ on $(K_{X_0})^\delta$ preserves the 
deformation quantization of the symplectic manifold $\left((K_{X_0})^\delta,\, 
\widetilde{\Omega}'_K\big\vert_{(K_{X_0})^\delta}\right)$, it follows immediately that the above deformation quantization 
of the symplectic manifold $$\left((K_{X_0})^\delta\cap K^{\delta}_0,\, {\Omega}'_K\big\vert_{(K_{X_0})^\delta\cap 
K^{\delta}_0}\right)$$ produces a deformation quantization of the above symplectic manifold
$\widehat{K^{\delta}_0/{\mathbb S}_\delta}$ equipped with the symplectic form constructed using
${\Omega}'_K\big\vert_{(K_{X_0})^\delta\cap 
K^{\delta}_0}$. Consequently, we obtain a deformation quantization of the Zariski open subset
$$
\widehat{\mathcal U}_S\, :=\, \mathcal{U}_S\cap \widehat{K^{\delta}_0/{\mathbb S}_\delta}\, \subset\,
\mathcal{U}_S
$$
equipped with the algebraic symplectic form $\widehat{\Omega}_S\,:=\,
{\Omega}_S\big\vert_{\widehat{\mathcal U}_S}$, where ${\Omega}_S$ and ${\mathcal U}_S$
are as in \eqref{fo1} and \eqref{e6} respectively.

Define
$$
\widehat{\mathcal U}_M\, :=\, \Phi^{-1}(\widehat{\mathcal U}_S)\, \subset\, {\mathcal U}_M\, ,
$$
where $\Phi$ is the isomorphism in \eqref{Ph}. Let
\begin{equation}\label{e10}
\widehat{\Phi}\, :=\, \Phi\big\vert_{\widehat{\mathcal U}_M}\, \, :\,\,
\widehat{\mathcal U}_M \, \stackrel{\sim}{\longrightarrow}\, \widehat{\mathcal U}_S
\end{equation}
be the isomorphism obtained by restricting $\Phi$.

The restriction of the symplectic form $\Omega_M$ in \eqref{fo2} to the above Zariski open subset
$\widehat{\mathcal U}_M\, \subset\, {\mathcal U}_M$ will be denoted by $\widehat{\Omega}_M$. From \eqref{e8}
it follows immediately that
\begin{equation}\label{e9}
\widehat{\Phi}^* \widehat{\Omega}_S\,\,=\,\,\widehat{\Omega}_M\, .
\end{equation}

Using $\beta$, we constructed above a deformation quantization of the symplectic manifold $\left(\widehat{\mathcal 
U}_S,\, \widehat{\Omega}_S\right)$. In view of \eqref{e9} this produces a deformation quantization of the symplectic 
manifold $(\widehat{\mathcal U}_M,\, \widehat{\Omega}_M)$. To describe this deformation quantization
explicitly, let
$$
\Psi \, :\, {\mathcal A}(\widehat{\mathcal U}_M)\, \longrightarrow\, {\mathcal A}(\widehat{\mathcal U}_S)
$$
be the isomorphism that sends any $\sum_{i=0}^\infty h^if_i\, \in\, {\mathcal A}(\widehat{\mathcal U}_M)$,
where $f_i\, \in\, {\mathcal H}(\widehat{\mathcal U}_M)$, to
$$\sum_{i=0}^\infty h^i(f_i\circ \widehat{\Phi}^{-1})\, \in\, {\mathcal A}(\widehat{\mathcal U}_S)\, .$$
Now for any $f_1,\, f_2\, \in\, {\mathcal A}(\widehat{\mathcal U}_M)$, define
$$
f_1\star f_2\, :=\, \Psi^{-1}(\Psi(f_1)\star \Psi(f_2))\, .
$$
This evidently defines a deformation quantization of the symplectic
manifold $\left(\widehat{\mathcal U}_M,\, \widehat{\Omega}_M\right)$.

The above construction is summarized in the following theorem:

\begin{theorem}\label{thm1}
Take a nonzero $1$-form $\beta\,\,\in\,\, H^0(X,\, K_X)\setminus\{0\}$. It produces a natural
deformation quantization of the Zariski open subset $\widehat{\mathcal U}_M$ of the moduli space
$\mathcal{M}_X(r)$ of semistable Higgs bundles on $X$ of rank $r$ and degree $r(g-1)+1$.
\end{theorem}

\section*{Data Availability}

Data sharing not applicable --- no new data generated.


\end{document}